\documentclass{article}
\usepackage{latexsym}
\usepackage{epsfig}
\usepackage{amsmath}
\usepackage{amssymb}
\usepackage{amsthm}
\newtheorem{theorem}{Theorem}[section]

\newtheorem{lemma}[theorem]{Lemma}

\theoremstyle{remark}

\theoremstyle{definition}

\begin{document}

\title{Optimal step length for the Newton method \\ near the minimum of a self-concordant function}

\author{Roland Hildebrand \thanks{%
Univ.\ Grenoble Alpes, CNRS, Grenoble INP, LJK, 38000 Grenoble, France
({\tt roland.hildebrand@univ-grenoble-alpes.fr}).}}

\maketitle

\begin{abstract}
In path-following methods for conic programming knowledge of the performance of the (damped) Newton method at finite distances from the minimizer of a self-concordant function is crucial for the tuning of the parameters of the method. The available bounds on the progress and the step length to be used are based on conservative relations between the Hessians at different points and are hence sub-optimal. In this contribution we use methods of optimal control theory to compute the optimal step length of the Newton method on the class of self-concordant functions, as a function of the initial Newton decrement, and the resulting worst-case decrease of the decrement. The exact bounds are expressed in terms of solutions of ordinary differential equations which cannot be integrated explicitly. We provide approximate numerical and analytic expressions which are accurate enough for use in optimization methods. As an application, the neighbourhood of the central path in which the iterates of path-following methods are required to stay can be enlarged, enabling faster progress along the central path during each iteration and hence fewer iterations to achieve a given accuracy.
\end{abstract}

\section{Introduction} \label{sec:intro}

The Newton method is a century-old second order method to find a zero of a vector field or a stationary point of a sufficiently smooth function. It is well-known that it is guaranteed to converge only in a neighbourhood of a solution, even on such well-behaved classes of objective functions as the self-concordant functions. In this neighbourhood the method converges quadratically. Farther away the step length has to be decreased to ensure convergence, leading to the \emph{damped} Newton method. Several rules have been proposed to choose the step length on different classes of cost functions or vector fields. Among these are line searches or path searches until some pre-defined condition is met \cite{Burdakov80,Ralph94}, strategies imported from gradient descent methods \cite[p.~37]{Nesterov18book}, and explicit formulas \cite[p.~24]{NesNem94}, \cite[p.~353]{Nesterov18book}. These rules provide sufficient conditions for convergence, but they are optimal only by the order of convergence \cite{Nesterov18book}.

A (damped) Newton step is given by the iteration
\[ x_{k+1} = x_k - \gamma_k(F''(x_k))^{-1}F'(x_k),
\]
where $\gamma_k = 1$ for a full step and $\gamma_k \in (0,1)$ if the step is damped. Here $F$ is the function to be minimized, assumed to be sufficiently smooth with a positive definite Hessian. A convenient measure to quantify the distance from the iterate $x_k$ to the minimizer is the \emph{Newton decrement} \cite{NesNem94}
\[ \rho_k = \sqrt{F'(x_k)^T(F''(x_k))^{-1}F'(x_k)} = ||F'(x_k)||^*_{x_k},
\]
where the norm of the gradient as measured in the local metric given by the inverse Hessian of $F$.

In this paper we investigate the following problems. Given $\rho_k$ and $\gamma_k$, what is the largest possible (worst-case) value $\rho_{k+1}$ of the decrement after the iteration? What is the optimal step-length $\gamma^*$, i.e., the value minimizing the upper bound on $\rho_{k+1}$ with respect to $\gamma_k$ for given $\rho_k$?

With the application to barrier methods for conic programming in mind we consider the class of \emph{self-concordant} functions. This class has been introduced in \cite{NesNem94}, because due to its invariance under affine coordinate transformations it is especially well suited for analysis in conjunction with the also affinely invariant Newton method. Self-concordant functions are locally strongly convex $C^3$ functions $F$ which satisfy the inequality
\[ |F'''(x)[h,h,h]| \leq 2(F''(x)[h,h])^{3/2}
\]
for all $x$ in the domain of definition and all $h$ in the tangent space at $x$. 
We shall now motivate the problem setting defined above by giving a short overview of the layout of path-following methods. For more details see \cite{NesNem94}. We consider the basic setup of a short-step primal path-following method.

\medskip

Consider the convex optimization problem $\min_{x \in X} \langle c,x \rangle$ with linear objective function, where $X \subset \mathbb R^n$ is a closed convex set. Suppose on the interior $D = X^o$ of the set a self-concordant function $F$ is defined which tends to infinity at the boundary of the set, $F|_{\partial X} = +\infty$. Such a situation arises within the framework of conic programming, in particular, linear, second-order conic, and semi-definite programming. Here $X$ is the feasible set of the primal problem, and $F$ is the restriction of the self-concordant barrier function to this set.

A primal path-following method generates a sequence $x_k \in D$ of iterates which lie in a neighbourhood of the \emph{central path} $x^*(\tau)$ of the problem and at the same time move along (follow) this path. Here for given $\tau > 0$ the point $x^*(\tau)$ is the unique minimizer of the composite function $F_{\tau}(x) = F(x) + \tau \langle c,x \rangle$ on $D$. For $\tau \to +\infty$ the central path and hence also the iterates $x_k$ converge to the optimal solution $x^*$ of the original problem.

In order to generate the next iterate $x_{k+1}$ from the current iterate $x_k$, a (damped) Newton step is made towards a target point $x^*(\tau_k)$ on the central path by attempting to minimize the composite function $F_{\tau_k}$. The performance of the method will depend on how the parameter $\tau_k$ of the target point and the step length $\gamma_k$ are chosen. The overall convergence rate of the method depends on how fast the sequence $\tau_k$ tends to $+\infty$. However, the larger $\tau_k$ at any given iterate is chosen, the farther the current point will lie from the target point, and the smaller the worst-case progress of the Newton iterate will be. The right choice of $\tau_k$ should optimally balance these two opposite trends.

Recall that the distance from the current iterate $x_k$ to the target point $x^*(\tau)$ is measured by the Newton decrement $\rho_k(\tau) = \sqrt{d_k^T(F''(x_k))^{-1}d_k}$, where $d_k = F'(x_k) + \tau c$ is the difference between the gradients of $F$ at the target and the current point. Note that $\rho_k^2(\tau)$ is a nonnegative quadratic function of $\tau$, with the leading coefficient given by $c^T(F''(x_k))^{-1}c$. The derivative $\frac{d\rho_k}{d\tau}$ is therefore bounded by $\sqrt{c^T(F''(x_k))^{-1}c}$.

The target point $x^*(\tau_k)$ is chosen such that $\tau_k$ is maximal under the constraint $\rho_k(\tau_k) \leq \overline{\lambda}$, where $\overline{\lambda}$ is a parameter of the method. The subsequent Newton step from $x_k$ to $x_{k+1}$ moves the iterate closer to the target point. More precisely, we have $\rho_{k+1}(\tau_k) \leq \underline{\lambda}$, where $\underline{\lambda}$ is some monotone function of $\overline{\lambda}$ which denotes an upper bound on the decrement at the next iterate. Since the preceding iteration from $x_{k-1}$ to $x_k$ also satisfied the constraint $\rho_{k-1}(\tau_{k-1}) \leq \overline{\lambda}$ and consequently $\rho_k(\tau_{k-1}) \leq \underline{\lambda}$, we may find a lower bound on the progress made by the iteration. Namely, from $\rho_k(\tau_k) = \overline{\lambda}$ and the bound on the derivative of $\rho_k$ we get $\tau_k - \tau_{k-1} \geq \frac{\overline{\lambda} - \underline{\lambda}}{\sqrt{c^T(F''(x_k))^{-1}c}}$.

It now becomes clear that we have to choose the parameter $\overline{\lambda}$ in order to maximize the difference $\overline{\lambda} - \underline{\lambda}$. The smaller $\underline{\lambda}$ as a function of $\overline{\lambda}$, the larger are both the optimal value $\lambda^*$ of $\overline{\lambda}$ and the difference $\lambda^* - \lambda_*$, where $\lambda_* = \underline{\lambda}(\lambda^*)$, and the faster the path-following method will progress. Recall that $\overline{\lambda}$ denotes the Newton decrement before the Newton step, while $\underline{\lambda}$ is an upper bound on the decrement after the Newton step. It is therefore of interest to use bounds $\underline{\lambda}(\overline{\lambda})$ on the decrement which are as tight as possible. This leads us exactly to the problem setting defined above.

\medskip

The currently used upper bound $\underline{\lambda}$ is based on the following estimate of the Hessian of a self-concordant function $F$ at a point $x$ near the current iterate $x_k$ \cite[Theorem 2.1.1]{NesNem94} (see also \cite[Theorem 5.1.7]{Nesterov18book})
\[ (1-||x-x_k||_{x_k})^2F''(x_k) \preceq F''(x) \preceq (1-||x-x_k||_{x_k})^{-2}F''(x_k).
\]
Here $||x-x_k||_{x_k} \in (0,1)$ is the distance between $x$ and $x_k$ as measured in the local metric at $x_k$.

Based on this estimate and assuming a full Newton step one obtains \cite[Theorem 5.2.2.1]{Nesterov18book}
\[ \underline{\lambda} = \left( \frac{\overline{\lambda}}{1-\overline{\lambda}} \right)^2,
\]
leading to $\lambda^* \approx 0.2291$ and $\lambda^* - \lambda_* \approx 0.1408$.

If instead one applies a damping coefficient $\gamma_k = \frac{1+\rho_k}{1+\rho_k+\rho_k^2}$, then one obtains the bound \cite[Theorem 5.2.2.3]{Nesterov18book}
\[ \underline{\lambda} = \overline{\lambda}^2\left(1+\overline{\lambda}+\frac{\overline{\lambda}}{1+\overline{\lambda}+\overline{\lambda}^2}\right).
\]
Then $\lambda^*$ increases to $\approx 0.2910$, the difference $\lambda^* - \lambda_*$ to $\approx 0.1638$, and the step length is given by $\frac{1+\lambda^*}{1+\lambda^*+(\lambda^*)^2} \approx 0.9384$.

\medskip

In this contribution we perform an exact worst-case analysis of the performance of the Newton iterate by reformulating it as an optimal control problem. The values of $\lambda^*$ for the full and the optimal damped Newton step turn out to be given by $\approx 0.3943$ and $\approx 0.4429$, respectively. The respective values of $\lambda^* - \lambda_*$ are given by $\approx 0.2184$ and $\approx 0.2300$, which represents an improvement by more than 40\%. The tight bound on $\rho_{k+1}$ and the optimal damping coefficient $\gamma_k$ as functions of $\rho_k$ are depicted in Fig.~\ref{fig:bounds}.

\begin{figure}
\centering
\includegraphics[width=13.85cm,height=5.05cm]{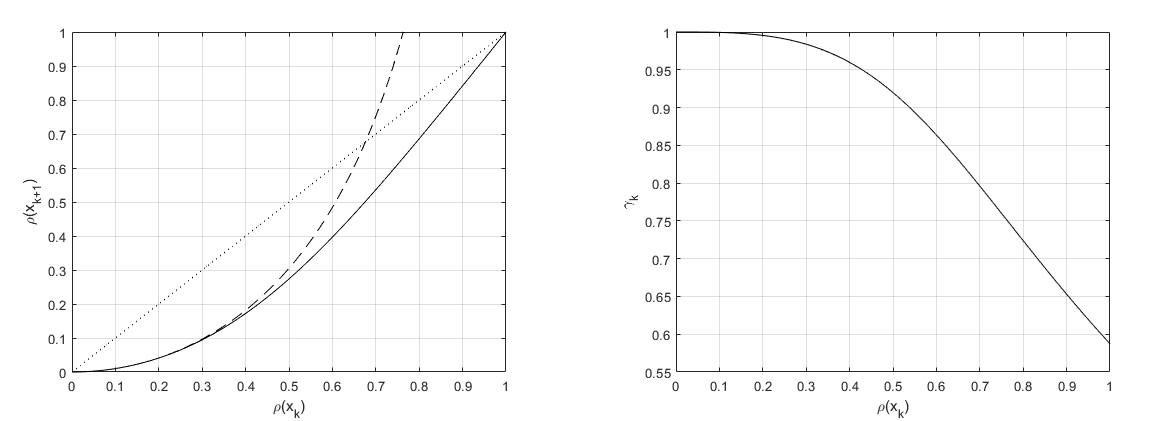}
\label{fig:bounds}
\caption{Upper bounds on the Newton decrement $\rho_{k+1}$ (left: dashed --- full Newton step, solid --- optimal damped Newton step) and optimal damping coefficient $\gamma_k$ (right) as a function of the current Newton decrement $\rho_k$.}
\end{figure}

The idea of analyzing iterative algorithms by optimization techniques is not new. An exact worst-case analysis of gradient descent algorithms by semi-definite programming has been performed in \cite{deKlerkGlineurTaylor17,TaylorHendrickxGlineur17}. In these papers an arbitrary finite number of steps is analyzed, but the function classes are such that the resulting optimization problem is finite-dimensional.

In \cite{deKlerkGlineurTaylor20accepted} a performance analysis for a single step of the Newton method on self-concordant functions is conducted. The class of self-concordant functions is, however, overbounded by a class of functions with Lipschitz-continuous Hessian, and the gradient at the next iterate is measured in the local norm of the previous iterate or in the (algorithmically inaccessible) local norm of the minimum. This yields also a finite-dimensional optimization problem. In \cite{GaoGoldfarb19} it is shown that the step size $\gamma_k = \frac{1}{1+\rho_k}$ proposed in \cite{NesNem94} maximizes a lower bound on the progress if the latter is measured in terms of the decrease of the function value. Using the techniques of this paper, one can show that this step length is actually optimal for this performance criterion\footnote{Unpublished joint work with Anastasia S. Ivanova (Moscow Institute of Physics and Technology).}.

In our case the properties of the class of self-concordant functions do not allow to obtain tight constraints on gradient and Hessian values at different points without taking into consideration all intermediate points, and the problem becomes infinite-dimensional. Our techniques, however, are borrowed from optimal control theory and nevertheless allow to obtain optimal bounds.

The remainder of the paper is structured as follows. In Section \ref{sec:1dim} we analyze the Newton iteration for self-concordant functions in one dimension. In this case the problem can be solved analytically. In Section \ref{sec:2dim} we generalize to an arbitrary number of dimensions. It turns out that the general case is no more difficult than the 2-dimensional one due to the rotational symmetry of the problem, and is described by the solutions of a Hamiltonian dynamical system in a 4-dimensional space. In Section \ref{sec:behaviour} we analyze the solutions of the Hamiltonian system. We establish that for small enough step sizes the 1-dimensional solutions are still optimal, while for larger step sizes (including the optimal one) the solution is more complicated and has no analytic expression anymore. In Section \ref{sec:damped} we minimize the upper bound on the Newton decrement $\rho_{k+1}$ with respect to the step length in order to obtain the optimal damped Newton step. In turns out that in this case the system simplifies to ordinary differential equations (ODE) on the plane. The optimal step length and the optimal bound on the decrement are then described in terms of solutions of these ODEs. In Section \ref{sec:numerical} we provide approximate numerical and analytic expressions for the upper bound on the decrement and for the optimal step size. In Section \ref{sec:path} we provide numerical values for the optimal parameters in path-following methods, both for the full and the optimally damped Newton step. We illustrate the obtained performance gain on a simple example. In Section \ref{sec:conclusion} we provide a game-theoretic interpretation of our results and summarize our findings.

\section{One-dimensional case} \label{sec:1dim}

In this section we analyze the Newton iteration in one dimension. Given a damping coefficient and the value of the Newton decrement at the current iterate, we would like to find the maximal possible value of the decrement at the next iterate. We reformulate this optimization problem as an optimal control problem. The solution of this problem is found by presenting an analytic expression for the Bellman function.

Let $F: I \to \mathbb R$ be a self-concordant function on an interval, i.e., a $C^3$ function satisfying $F'' > 0$, $|F'''| \leq 2(F'')^{3/2}$. Suppose the Newton decrement $\rho_k = \sqrt{\frac{F'(x_k)^2}{F''(x_k)}}$ at some point $x_k \in I$ equals $a \in (0,1)$. Fix a constant $\gamma \in (0,1)$ and consider the damped Newton step
\[ x_{k+1} = x_k - \gamma\frac{F'(x_k)}{F''(x_k)}.
\]
The Newton decrement at the next iterate, which we suppose to lie in $I$, is given by $\rho_{k+1} = \sqrt{\frac{F'(x_{k+1})^2}{F''(x_{k+1})}}$. Our goal in this section is to find the maximum of $\rho_{k+1}$ as a function of the parameters $a,\gamma$.

First of all, we may use the affine invariance of the problem setup to make some simplifications. We move the current iterate to the origin, $x_k = 0$, normalize the Hessian at the initial point to $F''(0) = 1$ by scaling the coordinate $x$, and possibly flip the real axis to achieve $F'(0) = -a < 0$. Then we get $x_{k+1} = a\gamma$. Introducing the functions $h = F''(x)$, $g = F'(x)$, we obtain the optimal control problem
\[ g' = h,\qquad h' = 2uh^{3/2},\qquad u \in U = [-1,1]
\]
with initial conditions
\[ g(0) = -a, \qquad h(0) = 1
\]
and objective function
\[ \sqrt{\frac{g(a\gamma)^2}{h(a\gamma)}} \to \sup.
\]

Replacing the state variable $g$ by $y = h^{-1/2}g$ and the independent variable $x$ by $t = h^{1/2}\cdot(x - a\gamma)$, we obtain $\frac{dt}{dx} = h^{1/2}\cdot(1+ut)$ and the problem becomes
\[ \dot y = \frac{1 - uy}{1 + ut},\qquad u \in U = [-1,1]
\]
with initial conditions $y(-a\gamma) = -a$ and objective function $|y(0)| \to \sup$. The variable $h$ becomes disconnected from the relevant part of the dynamics and can be discarded. If the control is bang-bang, i.e., $u$ is piece-wise constant with values in $\{-1,1\}$, then the dynamics can be integrated explicitly, with solutions
\[ -u\log|1 - uy| + const = u\log|1 + ut| \qquad \Rightarrow \quad (1 - uy)(1 + ut) = const.
\]

Following the principles of dynamic programming \cite{Bellman}, consider the \emph{Bellman function} $B(t_0,y_0)$ denoting the maximal objective value which can be achieved on a system trajectory $y(t)$ satisfying the initial condition $y(t_0) = y_0$. It satisfies the \emph{Bellman equation}
\[ \max_{u \in U}\frac{dB}{dt} = \max_{u \in U}\left( \frac{\partial B}{\partial y}\dot y + \frac{\partial B}{\partial t} \right) = 0
\]
with boundary condition $B(0,y) = |y|$. After a bit of calculation one obtains the solution
\[ B(t,y) = \left\{ \begin{array}{rcl} -y+t+ty,&\quad& y \leq \frac{2(-1+\sqrt{1+t^3})}{t^2}, \\ 4-y+t-ty-4\sqrt{(1-y)(1+t)},&& \frac{2(-1+\sqrt{1+t^3})}{t^2} \leq y \leq -t, \\ y-t-ty,&& y \geq -t\end{array} \right.
\]
on the domain $(t,y) \in [-1,0] \times \mathbb R$. The curve $y = -t$ is a switching curve, where the control $u$ switches from $+1$ to $-1$. The curve $y = \frac{2(-1+\sqrt{1+t^3})}{t^2}$ is a dispersion curve, there are two optimal trajectories with controls $u = \pm1$ emanating from the points of this curve, leading to final points $y(0)$ of different sign. The optimal synthesis of the system is depicted in Fig.~\ref{fig:1synthesis}. The optimal objective value is given by the value of $B$ at the initial point $(-a\gamma,-a)$.

\begin{figure}
\centering
\includegraphics[width=13.03cm,height=6.16cm]{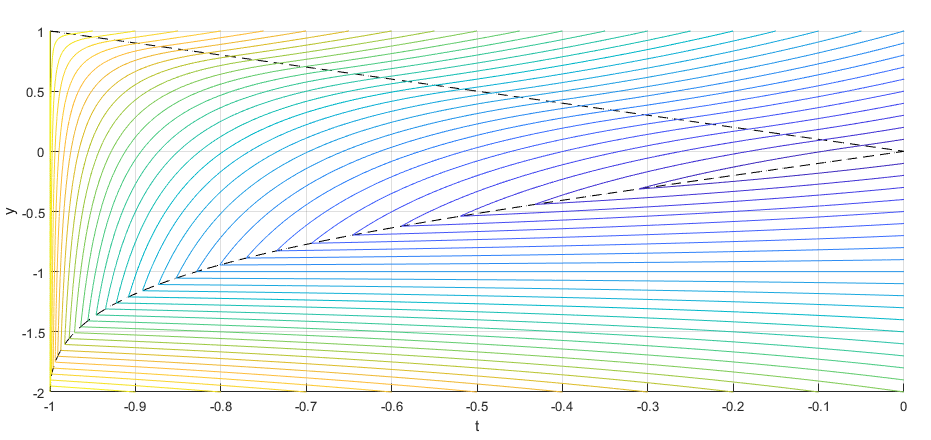}
\label{fig:1synthesis}
\caption{Optimal synthesis of the control problem modeling the one-dimensional case. The switching curve and the dispersion curve are dashed. The level curves of the Bellman function are at the same time the optimal trajectories of the system.}
\end{figure}

Let us now consider the result for the full Newton step. In this case $\gamma = 1$, and at the initial point is given by $(t,y) = (-a,-a)$. It hence lies between the dispersion curve and the switching curve, yielding the upper bound
\[ \rho_{k+1} \leq B(-a,-a) = 4-a^2-4\sqrt{1-a^2} = 4 - \rho_k^2 - 4\sqrt{1 - \rho_k^2}.
\]

For the optimally damped Newton step, we have to minimize $B(t,-a)$ with respect to $t \in [-a,0]$ for given $a$. The minimum is given by the intersection point $(\frac{2(1-\sqrt{1+a^3})}{a^2},-a)$ of the line $y = -a$ with the dispersion curve. Hence the optimal damping coefficient is given by $\gamma_k = \frac{2(\sqrt{1+\rho_k^3}-1)}{\rho_k^3}$. The corresponding upper bound on the Newton decrement is given by
\[ \rho_{k+1} \leq a+\frac{2(1-\sqrt{1+a^3})}{a^2}(1-a) = \frac{2(1-\rho_k)(1-\sqrt{1+\rho_k^3}) + \rho_k^3}{\rho_k^2}.
\]

\section{Reduction to a control problem in the general case} \label{sec:2dim}

In this section we perform a similar analysis for the case of self-concordant functions $F$ defined on $n$-dimensional domains. We reduce the problem to an optimal control problem in two state space dimensions. This results in a Hamiltonian system in 4-dimensional extended phase space.

First we simplify the dynamics as in the previous section. Introduce the vector-valued variable $g = F'$ and the matrix-valued variable $h = F'' = WW^T$, where $W$ is the lower-triangular factor of $h$ with positive diagonal. Let further ${\cal P}$ be the set of homogeneous cubic polynomials $p(x) = \sum_{i,j,k=1}^n p_{ijk}x_ix_jx_k$ which are bounded by 1 on the unit sphere $S^{n-1} \subset \mathbb R^n$. This is a compact convex set. The self-concordance condition then expresses the third derivatives of $F$ in the form
\[ \frac{\partial^3 F}{\partial x_i\partial x_j\partial x_k} = 2\sum_{r,s,t = 1}^n W_{ir}W_{js}W_{kt}p_{rst}, \qquad p \in {\cal P}.
\]
We shall need also the projection
\[ {\cal U} = \{ U \mid \exists\ p \in {\cal P}:\ U_{ij} = p_{ij1} \}
\]
of ${\cal P}$. The set ${\cal U}$ is a compact convex subset of the space of real symmetric matrices ${\cal S}^n$. Clearly it is overbounded by the set ${\cal U}' = \{ U \mid -I \preceq U \preceq I \}$. We shall also introduce the sets of lower-triangular matrices
\[ {\cal V} = \left\{ V \mid \frac{V+V^T}{2} \in {\cal U} \right\}, \qquad {\cal V}' = \left\{ V \mid \frac{V+V^T}{2} \in {\cal U}' \right\}
\]
which are compact and convex as well.

Using affine invariance, we may achieve the normalization $x_k = {\bf 0}$, $g({\bf 0}) = -ae_1$, $h({\bf 0}) = W({\bf 0}) = I$, $x_{k+1} = a\gamma e_1$. Here $a \in (0,1)$ is the Newton decrement $\rho_k$, $\gamma \in (0,1]$ the damping coefficient, and $e_1$ the first canonical basis vector. We consider the evolution of the variables $g,h,W$ only on the line segment joining $x_k$ and $x_{k+1}$, and may hence pass to a scalar variable $\tau \in [0,a\gamma]$, such that $x(\tau) = \tau e_1$ and $g(\tau) := g(x(\tau))$, $h(\tau) := h(x(\tau))$. The dynamics of the resulting control system can be written as
\[ \frac{dg}{d\tau} = he_1,
\]
\[ \frac{dh_{ij}}{d\tau} = 2\sum_{r,s,t = 1}^n W_{ir}W_{js}W_{1t}p_{rst} = 2W_{11}\sum_{r,s = 1}^n W_{ir}W_{js}U_{rs}, \ p \in {\cal P},\ U \in {\cal U}.
\]
It follows that
\[ \frac{dh}{d\tau} = \frac{dW}{d\tau}W^T + W\frac{dW^T}{d\tau} = 2W_{11}WUW^T
\]
and consequently
\[ W^{-1}\frac{dW}{d\tau} + \frac{dW^T}{d\tau}W^{-T} = 2W_{11}U,\qquad U \in {\cal U}.
\]
Since $W^{-1}\frac{dW}{d\tau}$ is lower-triangular, and $\frac{dW^T}{d\tau}W^{-T}$ is its transpose, we finally obtain
\[ \frac{dW}{d\tau} = W_{11}WV,\qquad V \in {\cal V}.
\]

We now replace $g$ by the variable $y = W^{-1}g$ and introduce a new independent variable $t = W_{11}\cdot(\tau - a\gamma)$. This variable then evolves in the interval $t \in [-a\gamma,0]$, and we have $\frac{dt}{d\tau} = W_{11}^2V_{11}\cdot(\tau - a\gamma) + W_{11} = W_{11}\cdot(tV_{11} + 1)$. The dynamics of the system becomes
\[ \dot y = \frac{1}{W_{11}(tV_{11} + 1)}(-W^{-1}(W_{11}WV)W^{-1}g + W^{-1}he_1) = \frac{-Vy + e_1}{tV_{11} + 1},\ V \in {\cal V}
\]
with initial condition $y(-a\gamma) = -ae_1$ and objective function
\[ \rho_{k+1} = \sqrt{g^T(x_{k+1})h^{-1}(x_{k+1})g(x_{k+1})} = ||y(0)|| \to \sup.
\]
The matrix-valued variable $W$ becomes disconnected and can be discarded.

For this problem the Bellman function cannot be presented in closed form, however, and we have to employ Pontryagins maximum principle \cite{PBGM62} to solve it. Introduce an adjoint vector-valued variable $p$, then the Pontryagin function and the Hamiltonian of the system are given by
\[ {\cal H}(t,y,p,V) = \frac{\langle p,-Vy + e_1 \rangle}{tV_{11} + 1},\qquad H(t,y,p) = \max_{V \in {\cal V}}\frac{\langle p,-Vy + e_1 \rangle}{tV_{11} + 1},
\]
respectively. The transversality condition is non-trivial at the end-point $t = 0$ and states that $p(0)$ equals the gradient $\frac{\partial ||y(0)||}{\partial y(0)} = \frac{y(0)}{||y(0)||}$ of the objective function. The optimal control problem is then reduced to the two-point boundary value problem
\begin{equation} \label{two_point_bdval}
\dot y = \frac{\partial H}{\partial p},\ \dot p = -\frac{\partial H}{\partial y},\ y(-a\gamma) = -ae_1,\ p(0) = \frac{y(0)}{||y(0)||}.
\end{equation}

Note that the problem setting is invariant with respect to orthogonal transformations of $\mathbb R^n$ which leave the distinguished vector $e_1$ invariant. Suppose that at some point on the trajectory the vectors $y,p$ span a plane containing $e_1$. Then at this point the component of the gradient of the Hamiltonian which is orthogonal to the plane vanishes due to symmetry, and the derivatives of $y,p$ are also contained in this plane. Therefore these variables will remain in this plane along the whole trajectory. Since at the end-point $t = 0$ the vectors $y,p$ are linearly dependent, we may assume without loss of generality that $y,p$ are for all $t$ contained in the plane spanned by the basis vectors $e_1,e_2$, or equivalently, that the dimension $n$ equals 2.

Then the set ${\cal P}$ is given by the set of bi-variate homogeneous cubic polynomials which are bounded by 1 on the unit circle, and can be expressed via the semi-definite representable set of nonnegative univariate trigonometric polynomials. This allows to obtain an explicit description of the set ${\cal V}$. Its boundary is given by the matrices
\[  V = \pm\frac{1}{2\cos^3\xi} \begin{pmatrix}  \cos\phi(3\cos^2\xi - \cos^2\phi) & 0 \\ 2\sin\phi(\sin^2\xi - \sin^2\phi) & \cos\phi(\cos^2\xi - \sin^2\phi) \end{pmatrix}
\]
with $\xi \in [0,\frac{\pi}{3}]$, $|\phi| \leq \xi$. Recall that the set ${\cal V}$ is overbounded by the set ${\cal V}'$ (see Fig.~\ref{fig:sets}). Both sets share the circle
\[ {\cal C} = \left\{ V = \begin{pmatrix} \cos\phi & 0 \\ 2\sin\phi & -\cos\phi \end{pmatrix} \mid \phi\in [-\pi,\pi] \right\}.
\]

\begin{figure}
\centering
\includegraphics[width=13.33cm,height=5.21cm]{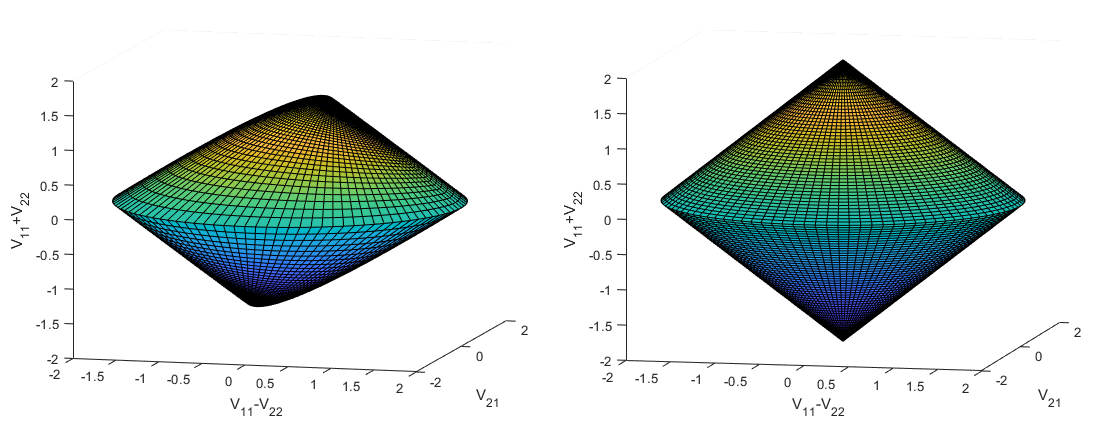}
\label{fig:sets}
\caption{True set ${\cal V}$ of controls $V$ (left) and overbounding set ${\cal V}'$ (right). The sharp circular edge of both bodies is the circle ${\cal C}$.}
\end{figure}

Note that the level sets of the Pontryagin function ${\cal H}$ are planes, because ${\cal H}$ is a fractional-linear function of $V$. Hence the maximum of ${\cal H}$ over a compact convex set is attained at an extreme point of this set. The maximum of ${\cal H}$ over the circle ${\cal C}$ can be computed explicitly and is given by the expression
\[ \max_{V \in {\cal C}}{\cal H} = \frac{p_1 + (p_1y_1 - p_2y_2)t + \sqrt{(p_1y_1 - p_2y_2 + p_1t)^2 + 4p_2^2y_1^2(1 - t^2)}}{1 - t^2}.
\]
Besides ${\cal C}$, the set ${\cal V}'$ has the extreme points $V = \pm I$, on which ${\cal H}$ evaluates to $\frac{p_1 \mp (p_1y_1+p_2y_2)}{1 \pm t}$. Hence we have $\max_{V \in {\cal V}'}{\cal H} = \max_{V \in {\cal C}}{\cal H}$ if
\[ \sqrt{(p_1y_1 - p_2y_2 + p_1t)^2 + 4p_2^2y_1^2(1 - t^2)} \geq -p_1t - p_1y_1 - p_2y_2 + 2p_2y_2t,
\]
\[ \sqrt{(p_1y_1 - p_2y_2 + p_1t)^2 + 4p_2^2y_1^2(1 - t^2)} \geq p_1t + p_1y_1 + p_2y_2 + 2p_2y_2t.
\]
These conditions then also imply $\max_{V \in {\cal V}}{\cal H} = \max_{V \in {\cal C}}{\cal H}$ and hence
\begin{equation} \label{Hamiltonian}
H(t,y,p) = \frac{p_1 + (p_1y_1 - p_2y_2)t + \sqrt{(p_1y_1 - p_2y_2 + p_1t)^2 + 4p_2^2y_1^2(1 - t^2)}}{1 - t^2}.
\end{equation}
It turns out a posteriori that the conditions are satisfied on the relevant solutions, and by virtue of the necessity of Pontryagins maximum principle for optimality we obtain the following result.

\begin{theorem} \label{thm:necessity}
Let $a,\gamma \in (0,1)$ be given. Then the upper bound on the Newton decrement $\rho_{k+1}$ after a damped Newton step with damping coefficient $\gamma$ and initial value of the decrement $\rho_k = a$ is given by the norm $||y(0)||$, where $(y(t),p(t))$, $y = (y_1,y_2)$, $p = (p_1,p_2)$, $t \in [-a\gamma,0]$, is a solution of boundary value problem \eqref{two_point_bdval} with Hamiltonian $H$ defined by \eqref{Hamiltonian}. \qed
\end{theorem}

In the next section we shall analyze the qualitative behaviour of the solutions of this Hamiltonian system.

\section{Qualitative behaviour of the solution curves} \label{sec:behaviour}

In this section we analyze the Hamiltonian system obtained in the previous section. It turns out that for small enough damping coefficients the trajectories corresponding to the 1-dimensional solution obtained in Section \ref{sec:1dim} are optimal. These trajectories are characterized by the condition that the variables $y_2,p_2$ are identically zero. If the initial point of the trajectory lies beyond a critical curve in the $(t,y_1)$-plane, however, all entries of $y,p$ participate in the dynamics. This section is devoted to the computation of the critical curve separating the two qualitatively different behaviours of the optimal trajectories.

As in the 1-dimensional case, the Bellman function $B(t,y)$ of the problem is constant on the trajectories of the Hamiltonian system and satisfies the boundary condition $B(0,y) = ||y||$. In order to construct it, we have to integrate the system in backward time with initial condition $p(0) = \frac{y(0)}{||y(0)||}$, launching a trajectory from every point of the $y$-plane.

In general, the projections on $y$-space of trajectories launched from different points may eventually intersect. In this case the trajectory with the maximal value of the Bellman function along it is retained. Therefore trajectories cease to be optimal at dispersion surfaces where they meet other trajectories with the same value of the Bellman function. In our case the plane $y_2 = 0$ acts as a dispersion surface by virtue of the symmetry $(y_1,y_2,p_1,p_2) \mapsto (y_1,-y_2,p_1,-p_2)$ of the system. Indeed, a trajectory launched from a point with $y_2(0) \not= 0$ and hitting the plane $y_2 = 0$ at some time will meet there with its image under the symmetry, which necessarily has the same value of the Bellman function. Thus the trajectories which are relevant for Theorem \ref{thm:necessity} are those which either completely evolve on the plane $y_2 = 0$, or which do not cross this plane in the time interval $(-a\gamma,0]$.

The first kind of trajectories correspond to the 1-dimensional system considered in Section \ref{sec:1dim} and are depicted in Fig.~\ref{fig:1synthesis}. In the regions between the dashed curves and the vertical axis they are given by
\begin{equation} \label{nominal}
y(t) = \frac{c + t}{1 - t}e_1,\quad p(t) = \frac{c(1 - t)}{|c|}e_1,
\end{equation}
where $c = y_1(0)$ is a parameter.

We now perturb trajectory \eqref{nominal} by launching it from a nearby point $y(0) = (c,\epsilon)$, and consider its evolution up to first order in $\epsilon$. This can be done by solving the linearization of ODE \eqref{two_point_bdval}. The coefficient matrix of the linearized system is given by
\[ \begin{pmatrix} \frac{\partial^2H}{\partial p\partial y} & \frac{\partial^2 H}{\partial p^2} \\ -\frac{\partial^2H}{\partial y^2} & -\frac{\partial^2H}{\partial y\partial p} \end{pmatrix} = \begin{pmatrix} \frac{1}{1 - t} & 0 & 0 & 0 \\ 0 & -\frac{1}{1 - t} & 0 & \frac{4y_1^2}{p_1(t + y_1)} \\ 0 & 0 & -\frac{1}{1 - t} & 0 \\ 0 & 0 & 0 & \frac{1}{1 - t} \end{pmatrix} = \begin{pmatrix} \frac{1}{1 - t} & 0 & 0 & 0 \\ 0 & -\frac{1}{1 - t} & 0 & \frac{4|c|(c + t)^2}{c(c + 2t - t^2)(1 - t)^2} \\ 0 & 0 & -\frac{1}{1 - t} & 0 \\ 0 & 0 & 0 & \frac{1}{1 - t} \end{pmatrix}.
\]
Here the first relation is obtained by setting $y_2 = p_2 = 0$ in the partial derivatives of $H$, the second one by inserting the values \eqref{nominal}. The linearized system has to be integrated with initial condition $\frac{\partial(c,\epsilon,\frac{c}{\sqrt{c^2+\epsilon^2}},\frac{\epsilon}{\sqrt{c^2+\epsilon^2}})}{\partial\epsilon}|_{\epsilon = 0} = (0,1,0,\frac{1}{|c|})$ at $t = 0$. It has the solution
\[ \delta y(t) = \delta y_2e_2,\qquad \delta p(t) = \frac{1}{|c|(1 - t)}e_2,
\]
where the scalar function $\delta y_2$ is a solution of the ODE
\[ \frac{d(\delta y_2)}{dt} = -\frac{1}{1 - t}\delta y_2 + \frac{4(c + t)^2}{c(c + 2t - t^2)(1 - t)^3}
\]
with initial condition $\delta y_2(0) = 1$. Integrating the ODE, we obtain
\begin{align*}
\delta y_2(t) &= \frac{-(c^2 + 5c + 16)(1-t)}{3c(c + 1)} + \frac{4(c + 2)}{c(c + 1)} - \frac{4}{c(1 - t)} + \frac{4(c + 1)}{3c(1 - t)^2} \\ & + \frac{4(1-t)\log\frac{c(1-t)^2}{c + 2t - t^2}}{c(c + 1)} + \frac{2(c + 2)(1-t)\log\frac{(c + 2t - t^2)(\sqrt{c+1} + 1)^2}{c(\sqrt{c+1} + 1 - t)^2}}{c(c + 1)^{3/2}}
\end{align*}
for $c > -1$, $\delta y_2 = \frac{4}{3(1-t)^2} - \frac{1-t}{3}$ for $c = -1$, and
\begin{align*}
\delta y_2(t) &= \frac{-(c^2 + 5c + 16)(1 - t)}{3c(c+1)} + \frac{4(c+2)}{c(c+1)} - \frac{4}{c(1-t)} + \frac{4(c+1)}{3c(1-t)^2} \\ & + \frac{4(1-t)\log\frac{c(1-t)^2}{c + 2t - t^2}}{c(c+1)} + \frac{4(c+2)(1-t)\left( \arctan\frac{1}{\sqrt{-1-c}} - \arctan\frac{1-t}{\sqrt{-1-c}} \right)}{c(-1-c)^{3/2}}
\end{align*}
for $c < -1$.

Setting the variable $\delta y_2$ to zero, we obtain a critical value of $t$ (a \emph{focal point}) on trajectory \eqref{nominal} beyond which other nearby trajectories of the system start to intersect the $y_2 = 0$ plane. At this point trajectory \eqref{nominal} ceases to be optimal. The ensemble of these critical points for all values $c \in \mathbb R$ forms a curve in the $(t,y_1)$-plane which marks the limit of optimality of the synthesis obtained in Section \ref{sec:1dim}. In order to obtain an expression for this curve, we have to use \eqref{nominal} to express $c$ as a function of $t,y_1$. Inserting this expression into the relation $\delta y_2 = 0$ yields a relation between $t$ and $y_1$.

For $c > -1$ this relation is given by
\[
(- Y^4T^6 + 4Y^4 - 3Y^2T^4 - 12y_1t)Y + 24T^2Y\log T + 6T(YT - 1)^2\log\frac{Y - T}{YT - 1} + 6T(YT + 1)^2\log\frac{YT + 1}{Y + T} = 0,
\]
where $Y = \sqrt{1+y_1}$, $T = \sqrt{1-t}$. For $c = -1$ we obtain the point $(t^*,-1)$ with $t^* = 1 - 2^{2/3} \approx -0.5874$, and for $c < -1$ we get
\[ (- Y^4T^3 + 4Y^4 + 3Y^2T^3 - 12T^2y_1t) + 12T^3 \left( 2\log T + \frac{(Y^2 - 1)(\arctan\frac{1}{Y} - \arctan\frac{T}{Y})}{Y} + \log\frac{Y^2 + 1}{Y^2 + T^2} \right) = 0,
\]
where $Y = \sqrt{-(1+y_1)(1-t)}$, $T = 1-t$.

In particular, both the switching curve and the dispersion curve of the 1-dimensional optimal synthesis lie beyond the critical curve and are not part of the full-dimensional solution (see Fig.~\ref{fig:critical}, left).

\begin{figure}
\centering
\includegraphics[width=13.07cm,height=5.09cm]{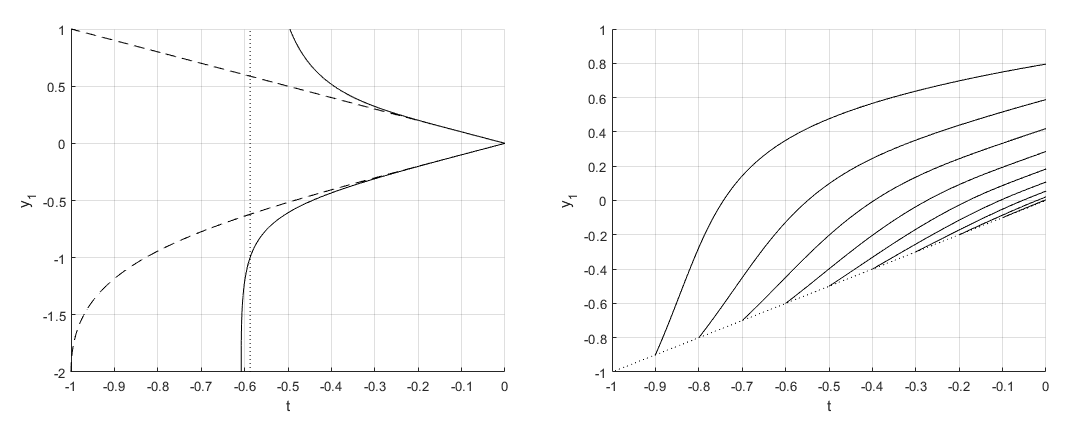}
\label{fig:critical}
\caption{Left: Critical curve marking the limit of optimality of the 1-dimensional solution (solid). For comparison the switching curve and the dispersion curve of the 1-dimensional optimal synthesis are also depicted (dashed). For $|y| \to +\infty$ the critical curve tends to the line $t = 1 - 2^{2/3}$ (dotted). Right: Projections on the $(t,y_1)$-plane of the trajectories corresponding to a full Newton step for different initial values $a$ of the Newton decrement. The dotted line is the locus of the initial points.}
\end{figure}

We obtain the following result.

\begin{theorem} \label{thm:critical}
Let $a,\gamma \in (0,1)$ be given, and suppose that the point $(t,y_1) = (-a\gamma,-a)$ lies to the right of the critical curve defined above. Then the upper bound on the Newton decrement $\rho_{k+1}$ after a damped Newton step with damping coefficient $\gamma$ and initial value of the decrement $\rho_k = a$ is given by $a-a\gamma+a^2\gamma$.
\end{theorem}

\begin{proof}
Since in this case the analysis in Section \ref{sec:1dim} is applicable, the optimal value of $\rho_{k+1}$ is given by $B(t,y_1) = -y_1+t+ty_1$.
\end{proof}

The trajectories with initial condition $y(-a\gamma) = -ae_1$ corresponding to a point $(t,y_1) = (-a\gamma,-a)$ beyond the critical curve can only be computed numerically by solving \eqref{two_point_bdval}. To each such initial condition there correspond two solutions which are taken to each other by the symmetry $(y_1,y_2,p_1,p_2) \mapsto (y_1,-y_2,p_1,-p_2)$. In particular, this will be the case both for the full Newton step and the optimal step length.

\section{Optimal step length and bounds for the damped Newton step} \label{sec:damped}

In this section we minimize the upper bound on $\rho_{k+1}$ with respect to the damping coefficient $\gamma$ for fixed initial values of the decrement $\rho_k = a$. The minimizer of this problem yields the optimal damping coefficient for the Newton iterate which leads to the largest guaranteed decrease of the decrement.

Technically, releasing $\gamma$ is equivalent to releasing the left end of the time interval on which the trajectory of the Hamiltonian system evolves, while leaving the initial state fixed. It is well known that the partial derivative with respect to time of the Bellman function, i.e., the objective achieved by the trajectory, equals the value of the Hamiltonian \cite{PBGM62}. Therefore we look for trajectories with starting points lying on the surface $H = 0$.

Let us first evaluate $H$ on the critical curve obtained in the previous section, more precisely on its arc between the lines $y_1 = -1$ and $y_1 = 0$. There we have $p_1 < 0$, $y_1 + t < 0$. Setting $y_2 = 0$, $p_2 = 0$, we obtain $H = \frac{p_1(1 + y_1)}{1 - t}$. Thus for $a < 1$ we have $H < 0$ and hence the optimal initial time instant $t$ is strictly smaller than the time instant defined by the critical curve. For $a = 1$, or equivalently $y_1 = -1$, the trajectory of the 1-dimensional system is optimal. As a consequence, for $a \to 1$ the optimal damping coefficient tends to $2^{2/3} - 1 \approx 0.5874$.

\begin{lemma}
The hyper-surface $H = 0$ is integral for the Hamiltonian system defined by \eqref{Hamiltonian}.
\end{lemma}

\begin{proof}
One easily computes
\[ \dot H = \frac{\partial H}{\partial t} = H \cdot \frac{p_1y_1 - p_2y_2 + p_1t + t\sqrt{(p_1y_1 - p_2y_2 + p_1t)^2 + 4p_2^2y_1^2(1 - t^2)}}{\sqrt{(p_1y_1 - p_2y_2 + p_1t)^2 + 4p_2^2y_1^2(1 - t^2)}(1 - t^2)},
\]
and hence if $H = 0$ somewhere on a trajectory, then $H \equiv 0$ everywhere on the trajectory.
\end{proof}

In particular, it follows that at the end-point $t = 0$ of the trajectory we also have $H = 0$. From \eqref{Hamiltonian} we then obtain by virtue of the transversality conditions $p(0) = \frac{y(0)}{||y(0)||}$ that
\[ H(0) = p_1 + \sqrt{(p_1y_1 - p_2y_2)^2 + 4p_2^2y_1^2} = \frac{y_1 + y_1^2 + y_2^2}{||y(0)||} = 0.
\]
The locus of the end-points in $y$-space is hence given by the circle $(y_1 + \frac12)^2 + y_2^2 = \frac14$.

From now on we assume without loss of generality that $y_2 \geq 0$ on the trajectory of the system.

Using the homogeneity of the dynamics with respect to the adjoint variable $p$ we may eliminate this variable altogether. On the surface $H = 0$ we have
\[ (y_1^2 - 1)p_1^2 - 2y_1y_2p_1p_2 + (4y_1^2 + y_2^2)p_2^2 = 0, \qquad \frac{p_1}{p_2} = \frac{y_1y_2 + \sqrt{- 4y_1^4 + 4y_1^2 + y_2^2}}{y_1^2 - 1},
\]
and consequently
\begin{align*}
\dot y_1 &= \frac{- p_1y_1^2 + p_2y_2y_1 + p_1}{p_1 + p_1y_1t - p_2y_2t} = \frac{\sqrt{- 4y_1^4 + 4y_1^2 + y_2^2}(-y_2(y_1 + t) + (y_1t + 1)\sqrt{- 4y_1^4 + 4y_1^2 + y_2^2})}{4y_1^4t^2 + 8y_1^3t + 4y_1^2 - y_2^2t^2 + y_2^2}, \\
\dot y_2 &= -\frac{4p_2y_1^2 - p_1y_1y_2 + p_2y_2^2}{p_1 + p_1y_1t - p_2y_2t} = \frac{\sqrt{- 4y_1^4 + 4y_1^2 + y_2^2}(4t y_1^3 + 4y_1^2 + y_2^2 - t y_2\sqrt{- 4y_1^4 + 4y_1^2 + y_2^2})}{4y_1^4t^2 + 8y_1^3t + 4y_1^2 - y_2^2t^2 + y_2^2}.
\end{align*}
A closer look reveals that the quotient of the two derivatives does not depend on $t$, and we obtain a planar dynamical system defined by the scalar ODE
\begin{equation} \label{planarODE}
\frac{dy_2}{dy_1} = \frac{\sqrt{4y_1^2(1 - y_1^2) + y_2^2} + y_1y_2}{1 - y_1^2}.
\end{equation}
We obtain the following result.

\begin{theorem}
Let $a \in (0,1)$ be given, and let $\sigma$ be the solution curve of ODE \eqref{planarODE} through the point $y_0 = (-a,0)$. Then the upper bound on the Newton decrement $\rho_{k+1}$ after a damped Newton step with optimal damping coefficient and initial value $\rho_k = a$ of the decrement is given by $||y^*|| = \sqrt{-y_1^*}$, where $y^* = (y_1^*,y_2^*)$ is the intersection point of the curve $\sigma$ with the circle centered on $(-\frac12,0)$ and with radius $\frac12$ in the upper half-plane $y_2 > 0$. \qed
\end{theorem}

The Riemann surfaces corresponding to the solutions of ODE \eqref{planarODE} in the complex plane have an infinite number of quadratic ramification points, and the equation is not integrable in closed form. The optimal bound on $\rho_{k+1}$ as a function of $\rho_k$ can be computed numerically and is depicted in Fig.~\ref{fig:bounds}, left.

In order to obtain the value of the optimal damping coefficient one also has to integrate the linear differential equation
\begin{equation} \label{damping_eq}
\begin{aligned}
\frac{dt}{dy_1} &= \frac{4y_1^4t^2 + 8y_1^3t + 4y_1^2 - y_2^2t^2 + y_2^2}{\sqrt{- 4y_1^4 + 4y_1^2 + y_2^2}(-y_2(y_1 + t) + (y_1t + 1)\sqrt{- 4y_1^4 + 4y_1^2 + y_2^2})} \\ &= \frac{y_2(y_1 + t) + (y_1t + 1)\sqrt{- 4y_1^4 + 4y_1^2 + y_2^2}}{(1 - y_1^2)\sqrt{- 4y_1^4 + 4y_1^2 + y_2^2}}.
\end{aligned}
\end{equation}

\begin{theorem}
Let $a \in (0,1)$ be given, and let $\sigma$ be the solution curve of ODE \eqref{planarODE} through the point $y_0 = (-a,0)$, intersecting the circle $(y_1+\frac12)^2 + y_2^2 = \frac14$ in the point $y^*$ in the upper half-plane $y_2 > 0$. Then the optimal damping coefficient $\gamma$ for the Newton step with initial value $\rho_k = a$ of the decrement is given by the value of $t(y_0)$, where $t(y)$ is the solution of ODE \eqref{damping_eq} along the curve $\sigma$ with initial value $t(y^*) = 0$. \qed
\end{theorem}

In order to compute the optimal value of $\gamma$ one hence has first to integrate ODE \eqref{planarODE} from $y_0$ to $y^*$ and then ODE \eqref{damping_eq} back from $y^*$ to $y_0$. The result is depicted on Fig.~\ref{fig:bounds}, right. The solution curves of the ODEs are depicted in Fig.~\ref{fig:ODEs}.

\begin{figure}
\centering
\includegraphics[width=13.55cm,height=5.09cm]{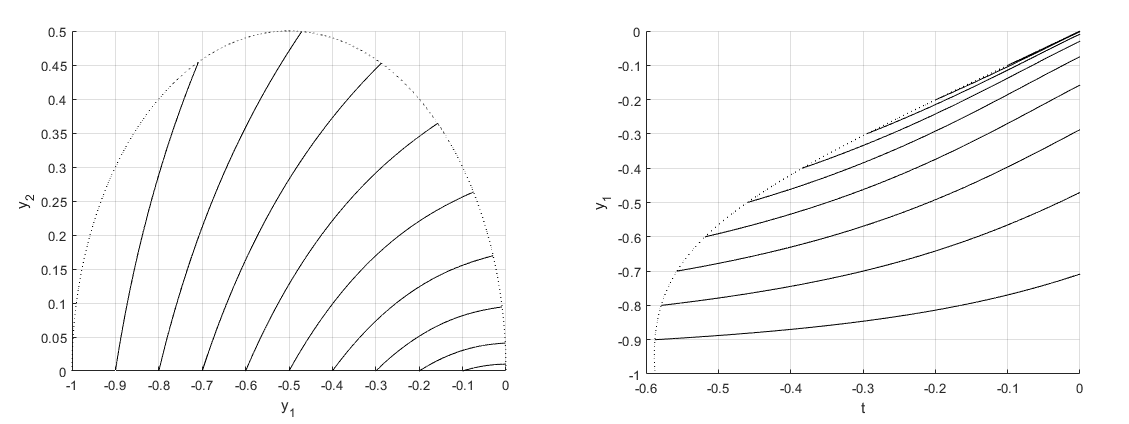}
\label{fig:ODEs}
\caption{Left: Solution curves of ODE \eqref{planarODE} between the line $y_2 = 0$ and the circle $(y_1+\frac12)^2 + y_2^2 = \frac14$ in the upper half-plane. Right: Solutions of ODE \eqref{damping_eq} on the curves depicted on the left part of the figure. The dotted curves are the locus of the end-points of the trajectories.}
\end{figure}

\section{Approximate values of the bound and the optimal step length} \label{sec:numerical}

As was established above, the worst-case Newton decrement after a Newton step with a given step length can only be computed by solving the two-point boundary value problem \eqref{two_point_bdval}. The optimal step length and the corresponding bound on the decrement can be computed by solving ODEs \eqref{planarODE} and \eqref{damping_eq}. In Table \ref{table:numerical} we provide numerical values for the upper bound $\underline{\lambda}$ of the decrement for the full Newton step and the optimal Newton step as a function of the decrement $\overline{\lambda}$ at the current point, and values for the optimal step length $\gamma^*$.

\begin{table}
\centering
{\footnotesize
\begin{tabular}{c|c|c|c}
$\overline{\lambda}$ & $\underline{\lambda}$, $\gamma = 1$ & $\underline{\lambda}$, $\gamma = \gamma^*$ & $\gamma^*$ \\
\hline
$0.02$ & $0.0004002129$ & $0.0004002020$ & $0.9999959049$ \\
$0.04$ & $0.0016030061$ & $0.0016027905$ & $0.9999672844$ \\
$0.05$ & $0.0025070365$ & $0.0025064667$ & $0.9999357402$ \\
$0.06$ & $0.0036140910$ & $0.0036128237$ & $0.9998883080$ \\
$0.08$ & $0.0064421367$ & $0.0064376124$ & $0.9997320984$ \\
$0.10$ & $0.0100985841$ & $0.0100863025$ & $0.9994703910$ \\
$0.12$ & $0.0145975978$ & $0.0145695698$ & $0.9990735080$ \\
$0.14$ & $0.0199560957$ & $0.0198993569$ & $0.9985103062$ \\
$0.15$ & $0.0229637111$ & $0.0228857123$ & $0.9981561975$ \\
$0.16$ & $0.0261938350$ & $0.0260886230$ & $0.9977481582$ \\
$0.18$ & $0.0333335449$ & $0.0331510886$ & $0.9967529721$ \\
$0.20$ & $0.0414011017$ & $0.0411009637$ & $0.9954892620$ \\
$0.22$ & $0.0504257486$ & $0.0499526517$ & $0.9939202785$ \\
$0.24$ & $0.0604403592$ & $0.0597204227$ & $0.9920082112$ \\
$0.25$ & $0.0658302428$ & $0.0649521741$ & $0.9909114838$ \\
$0.26$ & $0.0714817517$ & $0.0704180523$ & $0.9897144751$ \\
$0.28$ & $0.0835910576$ & $0.0820584243$ & $0.9870000911$ \\
$0.30$ & $0.0968141551$ & $0.0946530992$ & $0.9838261704$ \\
$0.32$ & $0.1112021742$ & $0.1082118504$ & $0.9801545078$ \\
$0.34$ & $0.1268120901$ & $0.1227421781$ & $0.9759482831$ \\
$0.35$ & $0.1350948306$ & $0.1303732829$ & $0.9736337787$ \\
$0.36$ & $0.1437074168$ & $0.1382488115$ & $0.9711728663$ \\
$0.38$ & $0.1619590244$ & $0.1547332145$ & $0.9657967079$ \\
$0.40$ & $0.1816461018$ & $0.1721931153$ & $0.9597922905$ \\
$0.42$ & $0.2028572983$ & $0.1906220829$ & $0.9531371033$ \\
$0.44$ & $0.2256920826$ & $0.2100091752$ & $0.9458145937$ \\
$0.45$ & $0.2377527563$ & $0.2200572997$ & $0.9418997667$ \\
$0.46$ & $0.2502623689$ & $0.2303386829$ & $0.9378150396$ \\
$0.48$ & $0.2766944756$ & $0.2515899946$ & $0.9291362838$ \\
$0.50$ & $0.3051314992$ & $0.2737375986$ & $0.9197842716$ \\
$0.52$ & $0.3357362119$ & $0.2967512350$ & $0.9097733400$ \\
$0.54$ & $0.3686946267$ & $0.3205961990$ & $0.8991262221$ \\
$0.55$ & $0.3861220234$ & $0.3328184758$ & $0.8935734256$ \\
$0.56$ & $0.4042204205$ & $0.3452337887$ & $0.8878737471$ \\
$0.58$ & $0.4425604721$ & $0.3706218774$ & $0.8760542409$ \\
$0.60$ & $0.4840018685$ & $0.3967155859$ & $0.8637126540$ \\
$0.62$ & $0.5288808639$ & $0.4234680202$ & $0.8508994659$ \\
$0.64$ & $0.5775944788$ & $0.4508310379$ & $0.8376694301$ \\
$0.65$ & $0.6035336028$ & $0.4647263175$ & $0.8309160437$ \\
$0.66$ & $0.6306157177$ & $0.4787560083$ & $0.8240802319$ \\
$0.68$ & $0.6885138337$ & $0.5071945318$ & $0.8101911337$ \\
$0.70$ & $0.7519817648$ & $0.5360990922$ & $0.7960616771$ \\
$0.72$ & $0.8218739745$ & $0.5654236219$ & $0.7817504964$ \\
$0.74$ & $0.8992597480$ & $0.5951239679$ & $0.7673142876$ \\
$0.75$ & $0.9411738550$ & $0.6101018892$ & $0.7600662717$ \\
$0.76$ & $0.9855000696$ & $0.6251582547$ & $0.7528069572$ \\
$0.78$ & $1.0823615911$ & $0.6554871449$ & $0.7382789623$ \\
$0.80$ & $1.1921910478$ & $0.6860740057$ & $0.7237768386$ \\
$0.82$ & $1.3181923462$ & $0.7168849922$ & $0.7093429029$ \\
$0.84$ & $1.4648868491$ & $0.7478890580$ & $0.6950151115$ \\
$0.85$ & $1.5479540249$ & $0.7634545483$ & $0.6879016840$ \\
$0.86$ & $1.6389206049$ & $0.7790579083$ & $0.6808270495$ \\
$0.88$ & $1.8505789254$ & $0.8103659079$ & $0.6668080264$ \\
$0.90$ & $2.1168852322$ & $0.8417899549$ & $0.6529832527$ \\
$0.92$ & $2.4687193058$ & $0.8733093323$ & $0.6393740764$ \\
$0.94$ & $2.9700851395$ & $0.9049055446$ & $0.6259982580$ \\
$0.95$ & $3.3195687439$ & $0.9207272500$ & $0.6194025003$ \\
$0.96$ &                & $0.9365621489$ & $0.6128702685$ \\
$0.98$ &                & $0.9682645833$ & $0.6000015959$ \\
\end{tabular}
}
\caption{Upper bounds on the decrement for a full and an optimal Newton step and optimal damping coefficient.}
\label{table:numerical}
\end{table}

In Fig.~\ref{fig:critical}, right, the solution curves of problem \eqref{two_point_bdval} corresponding to the value $\gamma = 1$ (full Newton step) and different $a$ are depicted. The resulting bound on the decrement after the iteration, listed in column 2 of Table \ref{table:numerical}, is depicted in Fig.~\ref{fig:bounds}, left, as the dashed line. As an analytic approximation of the optimal bound on the interval $[0,\frac23]$ one may use the upper bound
\begin{equation} \label{approx_full_bound}
\tilde{\underline{\lambda}} = 1.01 \cdot \overline{\lambda}^2 + 1.02 \cdot \overline{\lambda}^4,
\end{equation}
which is tight up to an error of $0.013$.

For the optimal step length the upper bound $\underline{\lambda}$ is depicted in Fig.~\ref{fig:bounds}, left, as the solid line. An asymptotic analysis of ODE \eqref{planarODE} for small values of $a$ yields the expansion
$\underline{\lambda} = \overline{\lambda}^2 - \frac14\overline{\lambda}^4\log\overline{\lambda} + \left( \frac{\log 2}{2} - \frac{1}{16} \right)\overline{\lambda}^4 + o(\overline{\lambda}^5)$. On the whole interval $[0,1]$ the analytic expression
\begin{equation} \label{approx_opt_bound}
\tilde{\underline{\lambda}} = \overline{\lambda}^2 - 0.556 \cdot \overline{\lambda}^4 \cdot \log\overline{\lambda}
\end{equation}
yields an upper bound which is tight up to an error of $0.007$.

An asymptotic analysis of ODE \eqref{damping_eq} for small values of $a$ leads to the expansion $\gamma^* = 1 - \frac{\overline{\lambda}^3}{2} + \frac{\overline{\lambda}^4}{4} + O(\overline{\lambda}^5\log \overline{\lambda})$ of the optimal damping coefficient. The analytic approximation
\begin{equation} \label{approx_opt_step}
\gamma^* \approx 1 - 0.95\cdot\overline{\lambda}^3 + 0.53\cdot\overline{\lambda}^4
\end{equation}
has an error of $0.008$ on the whole interval $[0,1]$. Unlike the situation in one dimension, in multiple dimensions the upper bound on the Newton decrement is a smooth function of the damping coefficient. Therefore a small deviation from the optimal value of $\gamma$ will result only in an increase of the upper bound $\underline{\lambda}$ by a term of second order.

\section{Application to path-following methods} \label{sec:path}

In this section we demonstrate how the analysis of the Newton iterate derived above serves to tune the parameters of a path-following method in a way leading to a significant performance gain.

For ease of implementation we consider a basic primal short-step path-following method for a random dense semi-definite program (SDP). Although it is in general not competitive with long-step methods, it serves to demonstrate the performance gain from an optimized tuning of its parameters. We shall compare four different setups, which differ by the used step length $\gamma$ and the policy of updating the parameter $\tau_k$ on the central path (see Section \ref{sec:intro} for a description of the method). The updating policy is determined by the choice of the parameter $\overline{\lambda}$, which equals the decrement before the Newton step.

Let us compute this parameter for the full and the optimal Newton step, respectively. We have to maximize the difference $\overline{\lambda} - \underline{\lambda}(\overline{\lambda})$ with respect to $\overline{\lambda}$. A numerical analysis of the optimal upper bounds $\underline{\lambda}(\overline{\lambda})$ for $\gamma = 1$ and $\gamma = \gamma^*$ yields the following results.

In the case of a full Newton step the optimal parameters $\lambda^*,\lambda_*$ are approximately given by $0.394257$, $0.175841$, respectively. The maximal difference evaluates to $\lambda^* - \lambda_* \approx 0.2184159486$, which is a 55\% improvement over the respective value presented in Section \ref{sec:intro}.

For an optimally damped Newton step $\lambda^*,\lambda_*$ are approximately given by $0.442946$, $0.212945$, respectively. The maximal difference evaluates to $\lambda^* - \lambda_* \approx 0.2300010331$, which is a 40\% improvement with respect to the bound from the intermediate Newton step in \cite[Theorem 5.2.2.3]{Nesterov18book}. The optimal step length at $\overline{\lambda} = \lambda^*$ is given by $\gamma^* \approx 0.944679$.

The setups of the path-following methods use the following parameter values:
\begin{itemize}
\item traditional method with full step, $\gamma = 1$, $\overline{\lambda} = 0.2291$;
\item tight method with full step, $\gamma = 1$, $\overline{\lambda} = 0.394257$;
\item traditional method with intermediate step, $\gamma = 0.9384$, $\overline{\lambda} = 0.2910$;
\item tight method with optimal step, $\gamma = 0.944679$, $\overline{\lambda} = 0.442946$.
\end{itemize}

The progress along the central path for the four setups is shown in Fig.~\ref{fig:sdp}. Optimizing the parameters of the method leads to approximately halving the number of iterations in comparison with the most simple setup.

\begin{figure}
\centering
\includegraphics[width=12.40cm,height=5.05cm]{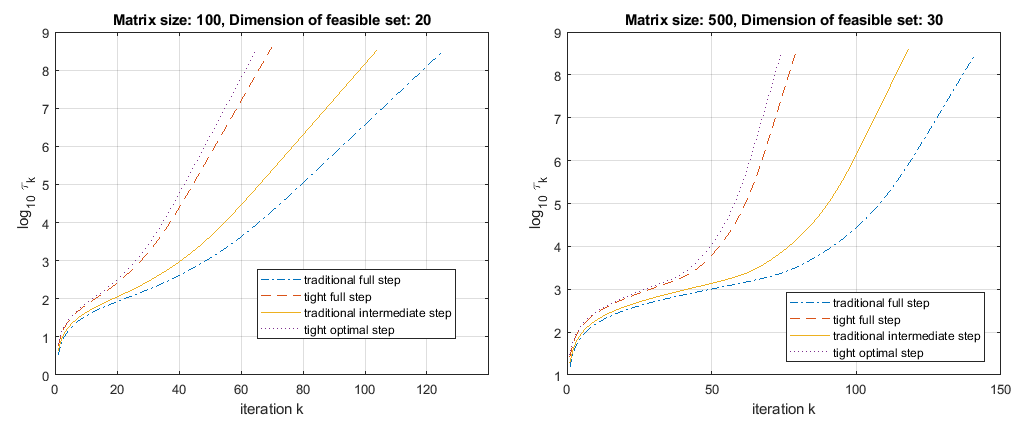}
\label{fig:sdp}
\caption{Performance of a short-step path-following method on a random SDP: growth of the parameter $\tau_k$ for the different choices of $\gamma,\overline{\lambda}$. }
\end{figure}

\section{Conclusion} \label{sec:conclusion}

We first furnish an interpretation of the results. Let us imagine the worst-case behaviour of the self-concordant function on the segment between the iterates $x_k,x_{k+1}$ as the response of an adversarial player to our choice of the damping coefficient. The goal of this player is to maximize the Newton decrement $\rho_{k+1}$ at the next iterate. Since the control which is at the disposal of the adversarial player affects the third derivative of the function, we can roughly assume that he manipulates the acceleration of the gradient.

First consider the case when the function is defined on an interval. Here the adversarial player has two different options. One is to maximally decelerate the gradient in order to prevent it from reaching zero at the end-point of the interval. This strategy will pay off more if we choose a smaller step length. The other strategy is to first maximally accelerate the gradient, in order to give it enough velocity to overshoot. At some point, corresponding to the crossing of the switching curve in Fig.~\ref{fig:1synthesis}, the gradient is again decelerated by decreasing the Hessian, because the effect of a smaller denominator $F''$ in the objective function overweighs the effect of a larger gradient $F'$, which enters in the numerator. This strategy pays off more if we choose a larger step length. Our optimal strategy will therefore be to choose that value of the damping coefficient which results in the same objective for both strategies of the adversarial player, i.e., we choose the initial point $(-a\gamma,-a)$ on the dispersion curve in Fig.~\ref{fig:1synthesis}.

In the case of a multi-dimensional domain of definition the adversarial player has more options. In addition to acceleration or deceleration of the gradient in the direction of movement, he may boost it in a perpendicular direction. Here he may choose this perpendicular direction arbitrarily, but once it is chosen, it is optimal to keep the acceleration vector in the plane spanned by the direction of movement and this particular direction. If the damping coefficient is large enough, more precisely if it corresponds to an initial point $(t,y_1)$ beyond the critical curve in Fig.~\ref{fig:critical}, left, the optimal strategy of the adversarial player is then indeed a mixture of boosts in the parallel and the perpendicular direction. Here the parallel component may be an acceleration or a deceleration, but the perpendicular component is always increased. For smaller damping coefficients the optimal strategy is a pure deceleration of the gradient in the direction of movement.

\medskip

Let us now summarize our findings. For a given value $\overline{\lambda}$ of the decrement at the initial point and given step length $\gamma$ the worst-case value $\underline{\lambda}$ of the decrement after the iteration can be computed by solving the two-point boundary value problem \eqref{two_point_bdval}. For a full Newton step, i.e., $\gamma = 1$, we provided both the numerical values of the bound on a grid (column 2 of Table \ref{table:numerical}) and the suboptimal analytic approximation \eqref{approx_full_bound}. To find the optimal value $\gamma^*$ of the damping coefficient and the corresponding minimal $\underline{\lambda}$ for a given initial value $\overline{\lambda}$ one needs to integrate two scalar ODEs. The numerical values of these quantities can be found in columns 4 and 3 of Table \ref{table:numerical}, respectively. In addition we provided the analytic approximations \eqref{approx_opt_step},\eqref{approx_opt_bound}. Formula \eqref{approx_opt_step} or any other reasonably simple approximation of the last column in Table \ref{table:numerical} can be used in any context when a self-concordant function is minimized by the Newton method.

Tuning of a path-following method using a predetermined step length $\gamma(\overline{\lambda})$ can be accomplished by maximizing the difference $\overline{\lambda} - \underline{\lambda}(\overline{\lambda})$ with respect to $\overline{\lambda}$. The parameters to be used are then the maximizer $\lambda^*$ and the corresponding damping coefficient $\gamma(\lambda^*)$. This computation needs to be done only once, thereafter the two obtained values can be encoded in the method. In Section \ref{sec:path} we computed these parameter values for two strategies, namely for the full Newton step $\gamma = 1$ and for the optimal Newton step, which minimizes $\underline{\lambda}$ with respect to $\gamma$.

\section*{Acknowledgments}

The paper was written in coronavirus quarantine at Moscow Institute of Physics and Technology. The author would like to thank the medical staff for having provided appropriate working conditions.

\bibliographystyle{plain}
\bibliography{interior_point,optimiz,convexity,opt_control}

\end{document}